\documentclass[a4paper,fleqn,11pt]{amsart}

\usepackage{mathrsfs}
\usepackage{mathptmx}
\usepackage{cite}
\usepackage{amssymb}
\usepackage{graphicx}

\DeclareMathAlphabet{\mathpzc}{OT1}{pzc}{m}{it}

\newtheorem{theorem}{Theorem}[section]

\newtheorem{claim}[theorem]{Claim}
\newtheorem{assertion}[theorem]{Assertion}
\newtheorem{corollary}[theorem]{Corollary}

\newtheorem{proposition}[theorem]{Proposition}
\theoremstyle{definition}
\newtheorem{definition}[theorem]{Definition}
\newtheorem{example}[theorem]{Example}
\theoremstyle{remark}
\newtheorem{remark}[theorem]{Remark}
\newtheorem{notation}[theorem]{Notation}

\numberwithin{equation}{section}

\allowdisplaybreaks[4]

\begin{document}

\title{Non-compact Newton boundary and Whitney equisingularity for non-isolated singularities} 

\author{Christophe Eyral and Mutsuo Oka}
\address{C. Eyral, Institute of Mathematics, Polish Academy of Sciences, \'Sniadeckich~8, 00-656 Warsaw, Poland}
\email{eyralchr@yahoo.com}
\address{M. Oka, Department of Mathematics, Tokyo University of Science, 1-3 Kagurazaka, Shinjuku-ku, Tokyo 162-8601, Japan}   
\email{oka@rs.kagu.tus.ac.jp}



\begin{abstract}
In an unpublished lecture note, J.~Brian\c con observed that if $\{f_t\}$ is a family of isolated complex hypersurface singularities such that the Newton boundary of $f_t$ is independent of $t$ and $f_t$ is non-degenerate, then the corresponding family of hypersurfaces $\{f_t^{-1}(0)\}$ is Whitney equisingular (and hence topologically equisingular). A first generalization of this assertion to families with non-isolated singularities was given by the second author under a rather technical condition. In the present paper, we give a new generalization under a simpler condition.
\end{abstract}

\maketitle

\markboth{C. Eyral and M. Oka}{Non-compact Newton boundary and Whitney equisingularity for non-isolated singularities} 

\section{Introduction}\label{intro}

Let $(t,\mathbf{z}):=(t,z_1,\ldots, z_n)$ be coordinates for $\mathbb{C}\times \mathbb{C}^n$, let $U$ be an open neighbourhood of $\mathbf{0}\in\mathbb{C}^n$ and $D$ be an open disc centered at $0\in\mathbb{C}$; finally, let 
\begin{equation*}
f\colon (D\times U, D\times \{\mathbf{0}\}) \rightarrow (\mathbb{C},0),\ (t,\mathbf{z}) \mapsto f(t,\mathbf{z}),
\end{equation*} 
be a polynomial function. As usual, we write $f_t(\mathbf{z}):=f(t,\mathbf{z})$ and we denote by $V(f_t)$ the hypersurface in $U\subseteq\mathbb{C}^n$ defined by $f_t$. We are interested in the local structure of the singular loci of the hypersurfaces $V(f_t)$ at the origin $\mathbf{0}\in\mathbb{C}^n$ as the parameter $t$ varies from a ``small'' non-zero value $t_0\not=0$ to $t=0$. More precisely, we are looking for easy-to-check conditions on the members $f_t$ of the family $\{f_t\}$ that guarantee equisingularity (in a sense to be specified) for the corresponding family of hypersurfaces $\{V(f_t)\}$. 

In an unpublished lecture note \cite{B}, J.~Brian\c con made the following observation.

\begin{assertion}[Brian\c con]\label{thm-briancon}
Suppose that for all $t$ sufficiently small, the following three conditions are satisfied:
\begin{enumerate}
\item
$f_t$ has an isolated singularity at the origin $\mathbf{0}\in\mathbb{C}^n$;
\item
the Newton boundary $\Gamma(f_t;\mathbf{z})$ of $f_t$ at $\mathbf{0}$ with respect to the coordinates $\mathbf{z}$ is independent of $t$;
\item
$f_t$ is non-degenerate (in the sense of the Newton boundary as in \cite{K,O2}). 
\end{enumerate}
Then the family of hypersurfaces $\{V(f_t)\}$ is Whitney equisingular.
\end{assertion}

We say that a family $\{V(f_t)\}$ of (possibly non-isolated) hypersurface singularities is \emph{Whitney equisingular} if there exist a Whitney stratification of the hypersurface $V(f):=f^{-1}(0)$ in an open neighbourhood $\mathscr{U}$ of the origin $(0,\mathbf{0})\in\mathbb{C}\times\mathbb{C}^n$ such that the $t$-axis $\mathscr{U}\cap(D\times\{\mathbf{0}\})$ is a stratum.
By ``Whitney stratification'' we mean a Whitney stratification in the sense of \cite{GWPL}---that is, we do not require that the frontier condition holds. However, note that if $\mathscr{S}$ is  Whitney stratification of $\mathscr{U}\cap V(f)$ with the $t$-axis as a stratum, then so is the partition $\mathscr{S}^c$ consisting of the connected components of the strata of $\mathscr{S}$;  moreover, $\mathscr{S}^c$ satisfies the frontier condition (see \cite{GWPL} for details).
Whitney equisingularity is quite a strong form of equisingularity. 
Combined with the Thom-Mather first isotopy theorem (cf.~\cite{T,M,GWPL}), it implies topological equisingularity. Here, we say that the family $\{V(f_t)\}$ is \emph{topologically equisingular} if  for all sufficiently small $t$, there is an open neighbourhood $U_t\subseteq U$ of $\mathbf{0}\in\mathbb{C}^n$ together with a homeomorphism $\varphi_t\colon (U_t,\mathbf{0})\rightarrow (\varphi_t(U_t),\mathbf{0})$ such that $\varphi_t(V(f_0)\cap U_t)=V(f_t)\cap \varphi_t(U_t)$.
Note that a family of \emph{isolated} hypersurface singularities (as in Assertion \ref{thm-briancon}) is Whitney equisingular if and only if $V(f)\setminus (D\times\{\mathbf{0}\})$ is smooth and Whitney $(b)$-regular over $D\times\{\mathbf{0}\}$ in an open neighbourhood of the origin in $\mathbb{C}\times\mathbb{C}^n$. Here, it is worth to observe that, in general, even if the smooth part of $V(f)$ is Whitney $(b)$-regular along the $t$-axis, the family of hypersurfaces $\{V(f_t)\}$ may  fail to be topologically equisingular. The simplest example illustrating this phenomenon is due to O. Zariski \cite{Z} and is as follows.  
Consider the family defined by the polynomial function $f(t,z_1,z_2):=t^2z_1^2-z_2^2$. The singular locus of $V(f)$ consists of two lines, namely the $t$-axis and the $z_1$-axis. Clearly, the smooth part of $V(f)$ is Whitney $(b)$-regular along the $t$-axis. However, there is no local ambient homeomorphism sending $(V(f_0),\mathbf{0})$ onto $(V(f_t),\mathbf{0})$. 

To conclude with Assertion \ref{thm-briancon}, let us mention that its proof is based on a famous theorem due to A.~G. Kouchnirenko \cite{K} and the second author \cite{O2}. This~theorem says that if $h(\mathbf{z})$ is a non-degenerate polynomial function with an isolated singularity at~$\mathbf{0}$, then its Milnor fibration and its Milnor number at~$\mathbf{0}$, as well as the local ambient topological type of the corresponding hypersurface $V(h)$ at~$\mathbf{0}$, are determined by the Newton boundary $\Gamma(h;\mathbf{z})$ of $h$ at~$\mathbf{0}$ with respect to the coordinates~$\mathbf{z}$. 

In \cite{O6}, the second author gave a generalization of Brian\c con's assertion to families of \emph{non-isolated} singularities under a relatively technical condition (so-called ``simultaneous IND-condition''). Essentially, the technical nature of this condition comes the fact that, in \cite{O6}, the question of Whitney equisingularity is treated in the general framework of complete intersection varieties. In the present paper, we restrict ourselves to the case of hypersurfaces, and we give a new generalization under a rather simple condition. As in \cite{O6}, when dealing with non-isolated singularities, we need to consider not only the compact faces of the Newton polygon (i.e., the faces involved in the usual Newton boundary) but also ``essential'' non-compact faces. The union of the compact and essential non-compact faces forms the \emph{non-compact Newton boundary}, which was considered not only in \cite{O6} but also in \cite{O1} to study the Milnor fibration and some geometric properties such as Thom's condition or the transversality of the nearby fibres of a polynomial function $h(\mathbf{z})$ with non-isolated singularities. (Actually, in \cite{O1}, the second author investigated the Milnor fibration and the Thom condition in the more general case of a ``mixed'' polynomial function $h(\mathbf{z},\overline{\mathbf{z}})$.) A crucial ingredient, introduced in \cite{O1}, to handle the essential non-compact faces is the ``local tameness.'' We shall see below that the local tameness---more precisely, a uniform version in $t$ of the local tameness---also plays a key role in our generalization of Brian\c con's assertion.

\section{Non-compact Newton boundary and local tameness}\label{sect-ncnblt}

In this section, we recall important definitions---due to A. G. Kouchnirenko \cite{K} and the second author \cite{O2,O1}---that will be used in this paper. A special emphasis will be given to the ``local tameness'' introduced by the second author in \cite{O1} and which will be crucial for our purpose.

Let $\mathbf{z}:=(z_1,\ldots, z_n)$ be coordinates for $\mathbb{C}^n$, let $U$ be an open neighbourhood of the origin $\mathbf{0}\in\mathbb{C}^n$, and let 
\begin{equation*}
h\colon(U,\mathbf{0})\to(\mathbb{C},0), \ \mathbf{z}\mapsto h(\mathbf{z})=\sum_\alpha c_\alpha\, \mathbf{z}^\alpha
\end{equation*}
be a polynomial function, where $\alpha:=(\alpha_1,\ldots,\alpha_n)\in\mathbb{N}^n$,  $c_\alpha\in\mathbb{C}$ and $\mathbf{z}^{\alpha}:=z_1^{\alpha_1}\cdots z_n^{\alpha_n}$. As usual, we write $V(h)$ for the hypersurface in $U\subseteq\mathbb{C}^n$ defined by $h$.
For any sub\penalty 5000 set $I\subseteq\{1,\ldots, n\}$, we set
\begin{align*}
& \mathbb{C}^I:=\{(z_1,\ldots, z_n)\in \mathbb{C}^n\ ;\ z_i=0 \mbox{ if } i\notin I\}, \\
& \mathbb{C}^{*I}:=\{(z_1,\ldots, z_n)\in \mathbb{C}^n\ ;\ z_i=0 \mbox{ if and only if } i\notin I\}.
\end{align*}
In particular, $\mathbb{C}^{\emptyset}=\mathbb{C}^{*\emptyset}=\{\mathbf{0}\}$ and $\mathbb{C}^{*\{1,\ldots,n\}}=(\mathbb{C}^*)^n$. (As usual, $\mathbb{C}^*:=\mathbb{C}\setminus \{\mathbf{0}\}$.)

The \emph{Newton polygon} of $h$ at $\mathbf{0}$ with respect to the coordinates $\mathbf{z}$ (denoted by $\Gamma_{\! +}(h;\mathbf{z})$) is the convex hull in $\mathbb{R}_+^n$ of the set
\begin{equation*}
\bigcup_{c_\alpha\not=0} (\alpha+\mathbb{R}_+^n),
\end{equation*}
where $\mathbb{R}_+^n=\{\mathbf{x}:=(x_1,\ldots, x_n)\in \mathbb{R}^n\ ;\ x_i\geq 0 \hbox{ for } 1\leq i\leq n\}$.
The \emph{Newton boundary} of $h$ at $\mathbf{0}$ with respect to $\mathbf{z}$ (denoted by $\Gamma(h;\mathbf{z})$) is the union of the compact faces of $\Gamma_{\! +}(h;\mathbf{z})$. 
For any system of weights $\mathbf{w}:=(w_1,\ldots,w_n)\in\mathbb{N}^n\setminus\{\mathbf{0}\}$, there is a linear map $\mathbb{R}^n\to\mathbb{R}$ given by
\begin{equation*}
\mathbf{x}:=(x_1,\ldots, x_n)\mapsto\sum_{1\leq i\leq n} x_i w_i.
\end{equation*}
Let $l_\mathbf{w}$ be the restriction of this map to $\Gamma_{+}(h;\mathbf{z})$, let $d_\mathbf{w}$ be the minimal value of $l_\mathbf{w}$, and let 
 $\Delta_\mathbf{w}$ be the face of $\Gamma_{+}(h;\mathbf{z})$ defined by the locus where $l_\mathbf{w}$ takes this minimal value. It is easy to see that
\begin{equation*}
d_\mathbf{w}:=\inf_{c_{\alpha}\not=0} \deg_{\mathbf{w}}(\mathbf{z}^\alpha),
\end{equation*}
where $\deg_{\mathbf{w}}(\mathbf{z}^\alpha)$ is the $\mathbf{w}$-degree of the monomial $\mathbf{z}^\alpha$, which is defined by
\begin{equation*}
\deg_{\mathbf{w}}(\mathbf{z}^\alpha):=\sum_{1\leq i\leq n} \alpha_i w_i=l_\mathbf{w}(\alpha).
\end{equation*}
Note that $\Delta_\mathbf{w}=\{\mathbf{x}\in \Gamma_{+}(h;\mathbf{z})\ ;\ l_\mathbf{w}(\mathbf{x})=d_\mathbf{w}\}$, and if $w_i>0$ for each $1\leq i\leq n$, then $\Delta_\mathbf{w}$ is a (compact) face of the Newton boundary $\Gamma(h;\mathbf{z})$.

\begin{definition}\label{def-mnbencf}
The \emph{non-compact Newton boundary} of~$h$ at $\mathbf{0}$ with respect to $\mathbf{z}$ (denoted by $\Gamma_{nc}(h;\mathbf{z})$) is obtained from the usual Newton boundary $\Gamma(h;\mathbf{z})$ by adding the ``essential'' non-compact faces. Here, a non-compact face $\Delta$ is said to be \emph{essential} if there are weights $\mathbf{w}:=(w_1,\ldots,w_n)\in\mathbb{N}^n\setminus \{\mathbf{0}\}$ such that the following two conditions hold:
\begin{enumerate}
\item[(i)]
$\Delta=\Delta_\mathbf{w}$ (i.e., $\Delta$ is the face defined by the locus where $l_\mathbf{w}$ takes its minimal value) and $h_{\mid \mathbb{C}^{I_\mathbf{w}}}=0$, where $I_\mathbf{w}:=\bigl\{i\in\{1,\ldots,n\}\ ;\, w_i=0\bigr\}$;
\item[(ii)]
for any $i\in I_\mathbf{w}$ and any point $\alpha\in\Delta$, the half-line $\alpha+\mathbb{R}_{+}\mathbf{e}_i$ is contained in $\Delta$, where $\mathbf{e}_i$ is the unit vector in the direction of the $x_i$-axis.
\end{enumerate}
\end{definition}

The set $I_\mathbf{w}$ does not depend on the choice of the weights $\mathbf{w}$. It is called the \emph{non-compact direction} of $\Delta$ and is denoted by $I_\Delta$.

\begin{example}[cf.~\cite{O1}]
If $h(z_1,z_2,z_3)=z_1^3+z_2^3+z_2z_3^2$, then the non-compact face $\Delta:=\overline{AC}+\mathbb{R}_+\mathbf{e}_3$ is essential, where $\overline{AC}$ is the edge with endpoints $A=(3,0,0)$ and $C=(0,1,2)$. Indeed, we can take $\mathbf{w}=(1,3,0)$. Then $\Delta=\Delta_{\mathbf{w}}$, $I_\Delta=\{3\}$ and $h(0,0,z_3)=0$ for any $z_3$. On the other hand, the non-compact face containing the edge $\overline{AB}$ (respectively, the edge $\overline{BC}$), where $B=(0,3,0)$, is not essential. Indeed, $h$ does not identically vanish neither on $\mathbb{C}^{\{1,2\}}$ nor on $\mathbb{C}^{\{2,3\}}$. See Figure \ref{figure1}.
\end{example}

\begin{remark}
Note that an essential non-compact face is not necessarily a maximal face of the Newton polygon. Indeed, in the above example, the $1$-dimensional non-compact face $\Xi:=C+\mathbb{R}_+\mathbf{e}_3$ is also essential. Indeed, we can take $\mathbf{w}'=(1,2,0)$. Then $\Xi=\Delta_{\mathbf{w}'}$, $I_\Xi=\{3\}$ and $h(0,0,z_3)=0$ for any $z_3$.
\end{remark}

\begin{figure}[t]
\includegraphics[width=16cm,height=5cm]{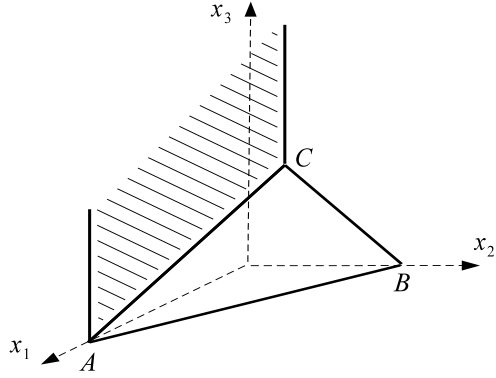}
\caption{Essential and non-essential non-compact faces}
\label{figure1}
\end{figure}

\begin{definition}\label{def-nd}
The function $h$ is said to be \emph{non-degenerate} if for any (compact) face $\Delta\subseteq \Gamma(h;\mathbf{z})$, the face function
\begin{equation*}
h_\Delta(\mathbf{z}):=\sum_{\alpha\in \Delta} c_\alpha\, \mathbf{z}^\alpha
\end{equation*}
has no critical point on $(\mathbb{C}^*)^n$.
\end{definition}

Let $\mathscr{I}_{nv}(h)$ (respectively, $\mathscr{I}_{v}(h)$) be the set of all subsets $I\subseteq \{1,\ldots,n\}$ such that $h_{\mid\mathbb{C}^I}\not=0$ (respectively, $h_{\mid\mathbb{C}^I}=0$). For any $I\in\mathscr{I}_{nv}(h)$ (respectively, any $I\in\mathscr{I}_{v}(h)$), the subspace $\mathbb{C}^I$ is called a \emph{non-vanishing} (respectively, a \emph{vanishing}) coordinates subspace.
For any $u_{i_1},\ldots, u_{i_m}\in\mathbb{C}^*$ ($m\leq n$), we set
\begin{equation*} 
\mathbb{C}^{*\{1,\ldots,n\}}_{u_{i_1},\ldots, u_{i_m}}:=
\bigl\{(z_1,\ldots,z_n)\in(\mathbb{C}^*)^n\, ;\, z_{i_j}=u_{i_j} \mbox{ for } 1\leq j\leq m\bigr\}.
\end{equation*}

\begin{definition}\label{definitionlocallytame}
Let $\Delta\subseteq\Gamma_{nc}(h;\mathbf{z})$ be an essential non-compact face, and let $\mathbf{w}=(w_1,\ldots,w_n)$ be a system of weights satisfying the conditions (i) and (ii) of Definition \ref{def-mnbencf}. Suppose that $I_\Delta:=I_{\mathbf{w}}=\{i_1,\ldots,i_m\}$ (i.e., $w_i=0$ if and only if $i\in\{i_1,\ldots,i_m\}$). We say that the face function 
\begin{equation*}
h_\Delta(\mathbf{z}):=\sum_{\alpha\in \Delta} c_\alpha\, \mathbf{z}^\alpha
\end{equation*}
is \emph{locally tame} if there exists a positive number $r(h_\Delta)>0$ such that for any non-zero complex numbers $u_{i_1},\ldots, u_{i_m}\in\mathbb{C}^*$ with 
\begin{equation*}
\vert u_{i_1}\vert^2+\cdots +\vert u_{i_m}\vert^2 < r(h_\Delta)^2,
\end{equation*} 
$h_\Delta$ has no critical point in $\mathbb{C}^{*\{1,\ldots,n\}}_{u_{i_1},\ldots, u_{i_m}}$ as a function of the $n-m$ variables $z_{i_{m+1}},\ldots,z_{i_n}$. (Here, $\{i_{m+1},\ldots,i_n\}=\{1,\ldots, n\}\setminus \{i_{1},\ldots,i_m\}$.)
We say that $h$ is \emph{locally tame along a vanishing coordinates subspace $\mathbb{C}^I$} if for any essential non-compact face $\Delta\subseteq\Gamma_{nc}(h;\mathbf{z})$ with $I_{\Delta}=I$, the face function $h_\Delta$ is locally tame. Finally, we say that $h$ is \emph{locally tame along the vanishing coordinates subspaces} if it is locally tame along $\mathbb{C}^I$ for any $I\in\mathscr{I}_v(h)$.
\end{definition}

\begin{notation}\label{notationrayon}
For any $I\in\mathscr{I}_v(h)$, let
\begin{equation*}
r_I(h):=\inf_{I_\Delta=I} r(h_\Delta).
\end{equation*} 
Also, let
\begin{equation*}
r_{nc}(h):=\inf_{I\in\mathscr{I}_v(h)} r_I(h).
\end{equation*}
\end{notation}

\begin{example}
If $h(z_1,z_2)=z_1^2z_2^3+z_1^3z_2^2+2z_1^2z_2^4$, then $\Gamma_{nc}(h;\mathbf{z})$ has two essential non-compact faces $\Delta_1:=A+\mathbb{R}_+\mathbf{e}_2$ and $\Delta_2:=B+\mathbb{R}_+\mathbf{e}_1$, where $A=(2,3)$ and $B=(3,2)$. Here, $I_{\Delta_1}=\{2\}$ and $I_{\Delta_2}=\{1\}$. For any $u_2\in\mathbb{C}^*$ with $\vert u_2\vert<1/2$, the function 
\begin{equation*}
z_1\mapsto h_{\Delta_1}(z_1,u_2)=z_1^2u_2^3+2z_1^2u_2^4
\end{equation*}
of the variable $z_1$ has no critical point on $\mathbb{C}^{*\{1,2\}}_{u_2}$. Thus the face function $h_{\Delta_1}$ is locally tame (we can take $r(h_{\Delta_1})=1/2$). Similarly, for any $u_1\in\mathbb{C}^*$, the function 
\begin{equation*}
z_2\mapsto h_{\Delta_2}(u_1,z_2)=u_1^3z_2^2
\end{equation*}
of the variable $z_2$ has no critical point on $\mathbb{C}^{*\{1,2\}}_{u_1}$, and hence the face function $h_{\Delta_2}$ is locally tame. Altogether, $h$ is locally tame along the vanishing coordinates subspaces.
\end{example}

\begin{example}
If $h(z_1,z_2,z_3)=z_1^2z_3^2-z_2^3z_3^2+z_3^3$, then $\Delta:=\overline{AB}+\mathbb{R}_+\mathbf{e}_1+\mathbb{R}_+\mathbf{e}_2$ is an essential non-compact face, where $\overline{AB}$ is the edge with endpoints $A=(2,0,2)$ and $B=(0,3,2)$. Here, $I_\Delta=\{1,2\}$. For any positive number $r>0$, there exist $u_1,u_2\in\mathbb{C}^*$ such that $\vert u_1\vert^2+\vert u_2\vert^2<r^2$ but the function
\begin{equation*}
z_3\mapsto h_{\Delta}(u_1,u_2,z_3)=u_1^2z_3^2-u_2^3z_3^2
\end{equation*}
of the variable $z_3$ has critical points on $\mathbb{C}^{*\{1,2,3\}}_{u_1,u_2}$. Indeed, the derivative of $z_3\mapsto h_{\Delta}(u_1,u_2,z_3)$ is zero along the curve $u_1^2-u_2^3=0$. It follows that $h_\Delta$ (and therefore $h$) is not locally tame.
\end{example}

\section{Admissible families and Whitney equisingularity}\label{sect-ult}

Let $(t,\mathbf{z}):=(t,z_1,\ldots, z_n)$ be coordinates for $\mathbb{C}\times \mathbb{C}^n$, let $U$ be an open neighbourhood of $\mathbf{0}\in\mathbb{C}^n$, let $D$ be an open disc centered at $0\in\mathbb{C}$, and let
\begin{equation*}
f\colon (D\times U, D\times \{\mathbf{0}\}) \rightarrow (\mathbb{C},0),\ (t,\mathbf{z}) \mapsto f(t,\mathbf{z}),
\end{equation*} 
be a polynomial function. (This notation implies that $f(D\times \{\mathbf{0}\})=\{0\}$.) As above, we write $f_t(\mathbf{z}):=f(t,\mathbf{z})$ and we denote by $V(f_t)$ the hypersurface in $U\subseteq\mathbb{C}^n$ defined by $f_t$.
In \S \ref{sub-admfam} of the present section, we introduce a condition (admissibility condition) that will guarantee Whitney equisingularity for families of non-isolated singularities. Roughly, a family $\{f_t\}$ is admissible if for all $t$ small enough, the non-compact Newton boundary $\Gamma_{nc}(f_t;\mathbf{z})$ is independent of $t$, the polynomial function $f_t$ is non-degenerate, and the radius $r_{nc}(f_t)$ which appear in Definition \ref{definitionlocallytame} and Notation \ref{notationrayon} is greater than or equal to a fixed positive number $\rho>0$. But before to go into further details, we first need to show that if the Newton boundary $\Gamma(f_t;\mathbf{z})$ is independent of $t$ and $f_t$ is non-degenerate for all small $t$, then, in a neighbourhood of the origin $\mathbf{0}\in\mathbb{C}^n$, the hypersurface $V(f_t)$ is smooth along $\mathbb{C}^{*I}$ for any $I\in\mathscr{I}_{nv}(f_t)$ and any $t$ small enough (cf.~\S \ref{sub-unismooth}).

\subsection{Smoothness along the non-vanishing coordinates subspaces}\label{sub-unismooth}

The following proposition is a uniform version of \cite[Chapter III, Lemma (2.2)]{O4} and \cite[Theorem 19]{O3}.

\begin{proposition}\label{lemmasmooth}
Suppose that for all $t$ sufficiently small, the following two conditions are satisfied:
\begin{enumerate}
\item
the Newton boundary $\Gamma(f_t;\mathbf{z})$ of $f_t$ at $\mathbf{0}$ with respect to the coordinates $\mathbf{z}$ is independent of $t$ (in particular, $\mathscr{I}_{nv}(f_t)$ is independent of $t$);
\item
the polynomial function $f_t$ is non-degenerate. 
\end{enumerate}
Then there exists a positive number $R>0$ such that for any $I\in\mathscr{I}_{nv}(f_0)$ and any $t$ sufficiently small, the set $V(f_t)\cap \mathbb{C}^{*I}\cap B_R$ is non-singular and intersects transversely with $S_r$ for any $r<R$, where $B_R$ (respectively, $S_r$) is the open ball (respectively, the sphere) with centre the origin $\mathbf{0}\in\mathbb{C}^n$ and radius $R$ (respectively, $r$).
\end{proposition}

\begin{proof}
As there are only finitely many subsets $I\in\mathscr{I}_{nv}(f_0)$, it suffices to show that for a fixed $I\in\mathscr{I}_{nv}(f_0)$, there is $R>0$ such that for any $t$ small enough, $V(f_t)\cap \mathbb{C}^{*I}\cap B_R$ is non-singular and intersects transversely with the sphere $S_r$ for any $r\leq R$. To simplify, we may assume that $I=\{1,\ldots,m\}$.

We start with the ``smoothness'' assertion. We argue by contradiction. Suppose that there exists a sequence $\{(t_N,\mathbf{z}_N)\}$ of points in $V(f)\cap(D\times\mathbb{C}^{*I})$ converging to $(0,\mathbf{0})$ and such that $\mathbf{z}_N$ is a critical point of the restriction of $f_{t_N}$ to $\mathbb{C}^{I}$. 
Then $(0,\mathbf{0})$ is in the closure of the set
\begin{align*}
W:=\biggl\{(t,\mathbf{z})\in D\times\mathbb{C}^{*I}\, ;\, {f_t}_{\mid \mathbb{C}^{I}}(\mathbf{z})=0\mbox{ and }\frac{\partial ({f_t}_{\mid \mathbb{C}^{I}})}{\partial z_{i}}(\mathbf{z})=0\mbox{ for }1\leq i\leq m\biggr\}.
\end{align*}
Therefore, by the curve selection lemma \cite{Milnor}, there is a real analytic curve 
\begin{align*}
(t(s),\mathbf{z}(s))=(t(s),z_1(s),\ldots,z_m(s),0,\ldots,0) 
\end{align*}
such that
$(t(0),\mathbf{z}(0))=(0,\mathbf{0})$ and $(t(s),\mathbf{z}(s))\in W$ for any $s\not=0$. For $1\leq i\leq m$, consider the Taylor expansions 
\begin{align*}
& t(s)=t_0s^{v}+\cdots,\\
& z_{i}(s)=a_{i}\, s^{w_{i}}+\cdots,
\end{align*}
where $t_0,a_{i}\not=0$ and $v,w_{i}>0$.  Here, the dots stand for the higher order terms. Let $\mathbf{a}:=(a_1,\ldots,a_m,0,\ldots,0)\in\mathbb{C}^{*I}$ and $\mathbf{w}:=(w_1,\ldots,w_m,0,\ldots,0)\in\mathbb{N}^n\setminus\{\mathbf{0}\}$, and let $\Delta$ be the face of $\Gamma\bigl({f_{t(s)}}_{\mid \mathbb{C}^{I}};\mathbf{z}\bigr)=\Gamma\bigl({f_0}_{\mid \mathbb{C}^{I}};\mathbf{z}\bigr)$  defined by the locus
where the map
\begin{align*}
\mathbf{x}:=(x_1,\ldots, x_m,0,\ldots,0)\in \Gamma({f_{t(s)}}_{\mid \mathbb{C}^{I}};\mathbf{z}) \mapsto 
\sum_{1\leq i\leq m} x_iw_i
\end{align*}
takes its minimal value $d$. For any $1\leq i\leq m$ and any $s\not=0$,
\begin{align*}
0 = \frac{\partial \bigl({f_{t(s)}}_{\mid \mathbb{C}^{I}}\bigr)}
{\partial z_{i}}(\mathbf{z}(s)) = 
\frac{\partial \bigl({f_{t(s)}}_{\mid \mathbb{C}^{I}}\bigr)_{\Delta}}
{\partial z_{i}}(\mathbf{a})\, s^{d-w_{i}}+\cdots,
\end{align*}
where $\bigl({f_{t(s)}}_{\mid \mathbb{C}^{I}}\bigr)_{\Delta}$ is the face function associated with ${f_{t(s)}}_{\mid \mathbb{C}^{I}}$ and $\Delta$.
It follows that
\begin{align*}
\frac{\partial \bigl({f_{0}}_{\mid \mathbb{C}^{I}}\bigr)_{\Delta}}{\partial z_{i}}(\mathbf{a})=0
\end{align*}
for any $1\leq i\leq m$. Therefore, $\mathbf{a}\in\mathbb{C}^{*I}$ is a critical point of $\bigl({f_{0}}_{\mid \mathbb{C}^{I}}\bigr)_{\Delta}\colon \mathbb{C}^{I}\to\mathbb{C}$. In particular, this implies that ${f_{0}}_{\mid \mathbb{C}^{I}}$ is \emph{not} non-degenerate as a function of the variables $z_1,\ldots, z_m$. This contradicts Proposition 7 of \cite{O3} which says that if a polynomial function ${f_{0}}$ is non-degenerate and if ${f_{0}}_{\mid \mathbb{C}^{I}}\not=0$, then ${f_{0}}_{\mid \mathbb{C}^{I}}$ must be non-degenerate as well. 

To prove the ``transversality'' assertion, we also argue by contradiction. Suppose that there exists a sequence $\{(t_N,\mathbf{z}_N)\}$ of points in $V(f)\cap(D\times \mathbb{C}^{*I})$ converging to $(0,\mathbf{0})$ and such that $V(f_{t_N})\cap \mathbb{C}^{*I}$ does not intersect the sphere $S_{\Vert z_N\Vert}$ transversely at $\mathbf{z}_N$. Then $(0,\mathbf{0})$ is in the closure of the set consisting of points $(t,\mathbf{z})\in D\times \mathbb{C}^{*I}$ such that
\begin{align*}
{f_t}_{\mid \mathbb{C}^{I}}(\mathbf{z})=0
\quad\mbox{and}\quad\mbox{grad} {f_t}_{\mid \mathbb{C}^{I}}(\mathbf{z})=\lambda \mathbf{z}
\mbox{ for } \lambda\in\mathbb{C}^*.
\end{align*}
Here, $\mbox{grad} {f_t}_{\mid \mathbb{C}^{I}}(\mathbf{z})$ is the gradient vector of ${f_t}_{\mid \mathbb{C}^{I}}$ at $\mathbf{z}$, that is,
\begin{align*}
\mbox{grad} {f_t}_{\mid \mathbb{C}^{I}}(\mathbf{z}):=
\biggl( \overline{\frac{\partial {f_t}_{\mid \mathbb{C}^{I}}}{\partial z_1} (\mathbf{z})}, \ldots, \overline{\frac{\partial {f_t}_{\mid \mathbb{C}^{I}}}{\partial z_m} (\mathbf{z})},0,\ldots,0 \biggr),
\end{align*}
where the bar stands for the complex conjugation.
Thus, by the curve selection lemma, we can find a real analytic curve 
\begin{align*}
(t(s),\mathbf{z}(s))=(t(s),z_1(s),\ldots,z_m(s),0,\ldots,0) 
\end{align*}
and a Laurent series $\lambda(s)$ such that:
\begin{enumerate}
\item[(i)]
$(t(0),\mathbf{z}(0))=(0,\mathbf{0})$;
\item[(ii)]
$(t(s),\mathbf{z}(s))\in D\times \mathbb{C}^{*I}$ for $s\not=0$;
\item [(iii)]
$f_{t(s)}(\mathbf{z}(s))=0$;
\item [(iv)]
$\mbox{grad} {f_{t(s)}}_{\mid \mathbb{C}^{I}}(\mathbf{z}(s))=\lambda(s) \mathbf{z}(s)$.
\end{enumerate}
Consider the Taylor expansions 
\begin{align*}
& t(s)=t_0s^{v}+\cdots,\\
& z_{i}(s)=a_{i}\, s^{w_{i}}+\cdots \ (1\leq i\leq m),
\end{align*}
where $t_0,\, a_{i}\not=0$ and $v,\, w_{i}>0$, and the Laurent expansion
\begin{align*}
\lambda(s)=\lambda_0s^\omega+\cdots,
\end{align*}
where $\lambda_0\not=0$.
Then define $\mathbf{a}$, $\mathbf{w}$, $d$ and $\Delta$ as above. 
By reordering, we may assume that $w_1=\cdots=w_k<w_j$ ($k<j\leq m$).
Then, by (iv), we have $d-w_1=\omega+w_1$ and
\begin{align}\label{exprder}
\overline{\frac{\partial \bigl({f_{0}}_{\mid \mathbb{C}^{I}}\bigr)_{\Delta}}
{\partial z_{i}}(\mathbf{a})}=\left\{
\begin{aligned}
& \lambda_0 a_i &&\mbox{for} && 1\leq i\leq k,\\
& 0 &&\mbox{for} && k< i\leq m.
\end{aligned}
\right.
\end{align}
Since the polynomial $\bigl({f_{0}}_{\mid \mathbb{C}^{I}}\bigr)_{\Delta}$ is weighted homogeneous with respect to the weights $\mathbf{w}$ and has weighted degree $d$, it follows from the Euler identity that 
\begin{align}\label{eigl25}
d\cdot \bigl({f_{0}}_{\mid \mathbb{C}^{I}}\bigr)_{\Delta}(\mathbf{a})=\sum_{1\leq i\leq m}w_i a_i\frac{\partial \bigl({f_{0}}_{\mid \mathbb{C}^{I}}\bigr)_{\Delta}}{\partial z_{i}}(\mathbf{a}).
\end{align}
As $f(t(s),\mathbf{z}(s))=0$ for any $s$, we have $\bigl({f_{0}}_{\mid \mathbb{C}^{I}}\bigr)_{\Delta}(\mathbf{a})=0$. Therefore, by combining (\ref{exprder}) and (\ref{eigl25}), we get a contradiction:
\begin{align*}
0=\sum_{1\leq i\leq m}w_i a_i\frac{\partial \bigl({f_{0}}_{\mid \mathbb{C}^{I}}\bigr)_{\Delta}}{\partial z_{i}}(\mathbf{a})=\bar\lambda_0 \sum_{1\leq i\leq k} w_i\vert a_i\vert^2\not=0.
\end{align*}
This completes the proof of Proposition \ref{lemmasmooth}.
\end{proof}

\begin{remark}
Under the same assumption as in Proposition \ref{lemmasmooth}, the second author showed in \cite{O5} that there also exists a positive number $R'>0$ such that for any $0<R''\leq R'$, the exists $\delta(R'')>0$ such that  for any $\eta\not=0$ with $\vert\eta\vert\leq\delta(R'')$ and any $r$ with $R''\leq r\leq R'$, the set $f_t^{-1}(\eta)\cap B_{R'}$ is non-singular and transversely intersects the sphere $S_{r}$ for all $t$ small enough. 
\end{remark}

\subsection{Admissible families and the main theorem}\label{sub-admfam}
Throughout this subsection, we assume that for all $t$ sufficiently small, the following two conditions hold:
\begin{enumerate}
\item[(I)]
the non-compact Newton boundary $\Gamma_{nc}(f_t;\mathbf{z})$ of $f_t$ at $\mathbf{0}$ with respect to $\mathbf{z}$ is independent of $t$ (in particular, $\mathscr{I}_{nv}(f_t)$ and $\mathscr{I}_{v}(f_t)$ are independent of $t$);
\item[(II)]
the polynomial function $f_t$ is non-degenerate and locally tame along the vanishing coordinates subspaces (i.e., locally tame along $\mathbb{C}^I$ for any $I\in\mathscr{I}_{v}(f_t)=\mathscr{I}_{v}(f_0)$); we denote by $r_{nc}(f_t)$ the corresponding positive number  defined in Definition \ref{definitionlocallytame} and Notation \ref{notationrayon}. 
\end{enumerate}

\begin{remark}
$\Gamma_{nc}(f_t;\mathbf{z})=\Gamma_{nc}(f_0;\mathbf{z}) \Leftrightarrow \Gamma(f_t;\mathbf{z})=\Gamma(f_0;\mathbf{z}) \Leftrightarrow \Gamma_{+}(f_t;\mathbf{z})=\Gamma_{+}(f_0;\mathbf{z})$.
\end{remark}

By Proposition \ref{lemmasmooth}, we know that there exists a positive number $R>0$ such that for any $I\in\mathscr{I}_{nv}(f_t)=\mathscr{I}_{nv}(f_0)$ and any $t$ small enough, $V(f_t)\cap \mathbb{C}^{*I}\cap B_R$ is non-singular.
It follows immediately that in a sufficiently small open neighbourhood $\mathscr{U}\subseteq D\times U$ of the origin of $\mathbb{C}\times\mathbb{C}^n$, the set
$V(f)\cap (\mathbb{C}\times \mathbb{C}^{*I})$ is non-singular for any $I\in\mathscr{I}_{nv}(f_t)$. Therefore, in such a neighbourhood, we can stratify $\mathbb{C}\times \mathbb{C}^n$ in such a way that the hypersurface $V(f):=f^{-1}(0)$ is a union of strata. More precisely, we consider the following three types of strata:
\begin{enumerate}
\item[$\cdot$]
$A_I:=\mathscr{U}\cap (V(f)\cap(\mathbb{C}\times \mathbb{C}^{*I}))$ 
for $I\in\mathscr{I}_{nv}(f_0)$;
\item[$\cdot$]
$B_I:=\mathscr{U}\cap((\mathbb{C}\times \mathbb{C}^{*I})\setminus (V(f)\cap(\mathbb{C}\times \mathbb{C}^{*I})))$ for $I\in\mathscr{I}_{nv}(f_0)$;
\item[$\cdot$]
$C_I:=\mathscr{U}\cap(\mathbb{C}\times \mathbb{C}^{*I})$ 
for $I\in\mathscr{I}_v(f_0)$.
\end{enumerate} 
The (finite) collection 
\begin{equation*}
\mathscr{S}:=\{A_I,B_I\, ;\, I\in\mathscr{I}_{nv}(f_0)\}\cup \{C_I\, ;\, I\in\mathscr{I}_{v}(f_0)\}
\end{equation*}
is a stratification (i.e., a partition into complex analytic submanifolds) of the set $\mathscr{U}\cap(\mathbb{C}\times\mathbb{C}^n)$ for which $\mathscr{U}\cap V(f)$ is a union of strata. Note that for $I=\emptyset$, which is an element of $\mathscr{I}_{v}(f_0)$, the stratum $C_{\emptyset}:=\mathscr{U}\cap(\mathbb{C}\times \mathbb{C}^{*\emptyset})$ of $\mathscr{S}$ is nothing but the $t$-axis $\mathscr{U}\cap(\mathbb{C}\times \{\mathbf{0}\})$.

\begin{remark}
A similar stratification but for a \emph{single} polynomial function $h(\mathbf{z})$ (not for a family) is already considered in \cite{O1} where the second author shows that if $h$ is non-degenerate and locally tame along the vanishing coordinates subspaces, then it satisfies Thom's $a_h$ condition with respect to this stratification.
\end{remark}

\begin{definition}
We say that the family $\{f_t\}$ is \emph{admissible} (at $t=0$) if it satisfies the conditions (I) and (II) above and if, furthermore, there exists a positive number $\rho>0$ such that for any sufficiently small $t$, 
\begin{equation*}
\inf\{R,r_{nc}(f_t)\} \geq \rho,
\end{equation*}
where $R$ is given by Proposition \ref{lemmasmooth}.
\end{definition}

In particular, if the family $\{f_t\}$ is admissible, then it is \emph{uniformly} locally tame along the vanishing coordinates subspaces---that is, $f_t$ is locally tame along $\mathbb{C}^I$ for any $I\in\mathscr{I}_{v}(f_0)$ and $r_{nc}(f_t)\geq\rho$ for all small $t$.

Here is our main result.

\begin{theorem}\label{mt2}
If the family of polynomial functions $\{f_t\}$ is admissible, then the canonical stratification $\mathscr{S}$ of $\mathscr{U}\cap(\mathbb{C}\times\mathbb{C}^n)$ described above is a Whitney stratification, and hence, the corresponding family of hypersurfaces $\{V(f_t)\}$ is Whitney equisingular.
\end{theorem}

We recall that a stratification of a subset of $\mathbb{C}^N$ is a \emph{Whitney stratification} if the closure $\bar S$ of each stratum $S$ and the complement $\bar S\setminus S$ are both analytic sets, and if for any pair of strata $(S_2,S_1)$ and any point $\mathbf{p}\in S_1\cap\bar S_2$, the stratum $S_2$ is Whitney $(b)$-regular over the stratum $S_1$ at the point $\mathbf{p}$. The latter condition means that  
for any sequences of points $\{\mathbf{p}_k\}$ in $S_{1}$, $\{\mathbf{q}_k\}$ in $S_{2}$ and $\{a_k\}$ in $\mathbb{C}$ satisfying:
\begin{enumerate}
\item[(i)]
$\mathbf{p}_k\to\mathbf{p}$ and $\mathbf{q}_k\to\mathbf{p}$;
\item[(ii)]
$T_{\mathbf{q}_k} S_{2}\to T$;
\item[(iii)]
$a_k(\mathbf{p}_k-\mathbf{q}_k)\to v$;
\end{enumerate}
we have $v\in T$. (As usual, $T_{\mathbf{q}_k} S_{2}$ is the tangent space to $S_{2}$ at $\mathbf{q}_k$.) For details, we refer the reader to [2].

\begin{remark}\label{remarkstratification}
Observe that if $M$ is a smooth manifold and $N\subseteq M$ is a closed smooth submanifold, then $M\setminus N$ is Whitney $(b)$-regular over $N$ at any point.
\end{remark}

Theorem \ref{mt2} will be proved in Section \ref{pmt2}. The proof will show that for \emph{isolated} singularities, if $\Gamma(f_t;\mathbf{z})$ is independent of $t$ and $f_t$ is non-degenerate for all small $t$, then the conclusions of Theorem \ref{mt2} still holds true without assuming the uniform local tameness (cf.~Remark \ref{rk-rbt}). In other words, Theorem \ref{mt2} includes Assertion \ref{thm-briancon} as a special case. 

\begin{remark}
In \cite[\S 8]{O6}, another Whitney stratification of $\mathscr{U}\cap(\mathbb{C}\times\mathbb{C}^n)$, with the $t$-axis as a stratum and such that $\mathscr{U}\cap V(f)$ is a union of strata, is constructed under a different assumption (so-called ``simultaneous IND-condition''). However this stratification is different from our. Especially, it has a larger number of strata.
\end{remark}

Combined with the Thom-Mather first isotopy theorem (cf.~\cite{T,M,GWPL}), Theorem \ref{mt2} implies the following result.

\begin{corollary}\label{mt}
If the family of polynomial functions $\{f_t\}$ is admissible, then the corresponding family of hypersurfaces $\{V(f_t)\}$ is topologically equisingular.
\end{corollary}

\begin{remark}
Topological equisingularity is also proved in \cite[Theorem (8.2)]{O6} under the simultaneous IND-condition.
\end{remark}

\begin{proof}[Proof of Corollary \ref{mt}]
By Theorem \ref{mt2}, $(\mathscr{U}\cap(\mathbb{C}\times \mathbb{C}^n),\mathscr{S})$ is a Whitney stratified set. Hence, by the Thom-Mather first isotopy theorem, it is topologically locally trivial (see, e.g., Theorem (5.2) and Corollary (5.5) of \cite{GWPL}). As the stratum $C_{\emptyset}$ is nothing but the $t$-axis, Corollary \ref{mt} follows.
\end{proof}

\nopagebreak[4]
\section{Proof of Theorem \ref{mt2}}\label{pmt2}

First of all, observe that if $I\subseteq J$, then $\mathbb{C}^{*I}$ is contained in the closure $\overline{\mathbb{C}^{*J}}$ of $\mathbb{C}^{*J}$. Moreover, if  $I\subseteq J$ and $J\in\mathscr{I}_v(f_0)$, then  $I\in\mathscr{I}_v(f_0)$ too. Therefore, to prove the theorem, it suffices to check that the Whitney $(b)$-regularity condition holds for all the pairs of strata satisfying one of the following three conditions:
\begin{enumerate}
\item 
$C_I \cap \overline{C_J}\not=\emptyset$ 
with $I\subseteq J$ and $I,J\in\mathscr{I}_{v}(f_0)$;
\item
$C_I \cap \overline{A_J}\not=\emptyset$ or $C_I \cap \overline{B_J}\not=\emptyset$
with $I\subseteq J$ and $I\in\mathscr{I}_v(f_0)$, $J\in\mathscr{I}_{nv}(f_0)$;
\item
$A_I \cap \overline{A_J}\not=\emptyset$,
$A_I \cap \overline{B_J}\not=\emptyset$ or
$B_I \cap \overline{B_J}\not=\emptyset$
with
$I\subseteq J$ and $I,J\in\mathscr{I}_{nv}(f_0)$.
\end{enumerate}
Except for the case of a pair of strata of the form $(A_J,C_I)$, the Whitney $(b)$-regularity condition immediately follows from Remark \ref{remarkstratification}. Thus, to prove our result, it suffices to show that for any $J\in\mathscr{I}_{nv}(f_0)$ and any $I\in\mathscr{I}_v(f_0)$, with $I\subseteq J$, $A_J:=\mathscr{U}\cap(V(f)\cap (\mathbb{C}\times\mathbb{C}^{*J}))$ is Whitney $(b)$-regular over $C_I:=\mathscr{U}\cap(\mathbb{C}\times\mathbb{C}^{*I})$ at any point $(\tau,\mathbf{q})=(\tau,q_1,\ldots,q_n)\in C_I \cap \overline{A_J}$ sufficiently close to the origin of $\mathbb{C}\times\mathbb{C}^n$. 
To simplify, without loss of generality, we may assume that $J=\{1,\ldots,n\}$ and $I=\{1,\ldots,m\}$ with $1\leq m\leq n-1$. (For $I=\emptyset$, see Remark \ref{rk-Ivide}.) In particular, $q_i\not=0$ if and only if $1\leq i\leq m$. 
By the curve selection lemma \cite{Milnor}, it is enough to show that the Whitney $(b)$-regularity condition holds along arbitrary real analytic paths 
\begin{align*}
& \gamma(s):=(t(s),\mathbf{z}(s)):=(t(s),z_1(s),\ldots,z_n(s))\\
& \tilde{\gamma}(s):=(\tilde t(s),\tilde{\mathbf{z}}(s)):=(\tilde t(s),\tilde z_1(s),\ldots,\tilde z_n(s))
\end{align*}
such that $\gamma(0)=\tilde{\gamma}(0)=(\tau,\mathbf{q})$, and $\tilde{\gamma}(s)\in C_I$ and $\gamma(s)\in A_J$ for $s\not=0$.
Consider the Taylor expansions (where, as above, the dots stand for the higher order terms):
\begin{align*}
& t(s)=\tau+b_0s+\cdots,\quad
z_i(s)=a_i s^{w_i}+b_i s^{w_i+1}+\cdots,\\
& \tilde t(s)=\tau+\tilde b_0 s+\cdots, \quad
\tilde z_i(s)=q_i +\tilde b_i s+\cdots,
\end{align*}
where $w_i=0$ and $a_i=q_i$ for $1\leq i\leq m$ while $w_i>0$ and $a_i\not=0$ for $i>m$. Note that, for any $s$, we have $\tilde z_i(s)=0$ if $i>m$.
Let
\begin{equation*}
\ell(s):=\overrightarrow{\tilde\gamma(s)\gamma(s)}
=(\ell_0(s),\ell_1(s),\ldots,\ell_n(s))
\end{equation*}
where
\begin{equation*}
\ell_i(s):=\left\{
\begin{aligned}
& (b_0-\tilde b_0)s+\cdots &&  \mbox{ for } i=0,\\
& (b_i-\tilde b_i)s+\cdots && \mbox{ for } 1\leq i\leq m,\\
& a_is^{w_i}+\cdots && \mbox{ for } m+1\leq i\leq n.
\end{aligned}
\right.
\end{equation*}
By reordering, we may suppose that
\begin{equation}\label{assumptionwi}
\begin{aligned}
w_{m+1}=\cdots=w_{m+m_1}<w_{m+m_1+1}=\cdots=w_{m+m_1+m_2}<\cdots\\
\cdots < w_{m+m_1+\cdots+m_{k-1}+1}=\cdots=w_{m+m_1+\cdots+m_{k}}=w_n,
\end{aligned}
\end{equation}
for some non-negative integers $m_1,\ldots,m_k\in\mathbb{N}$ with $m+m_1+\cdots+m_k=n$.
To show that the pair of strata $(A_J,C_I)$ satisfies the Whitney $(b)$-regularity condition at the point $(\tau,\mathbf{q})$, we have to prove that 
\begin{equation}\label{lfwbrc}
\lim_{s\to 0} \frac{\langle \ell(s),\mbox{grad} f(\gamma(s))\rangle}{\Vert \ell(s) \Vert\cdot \Vert \mbox{grad} f(\gamma(s)) \Vert}=0,
\end{equation}
where $\langle \cdot,\cdot\rangle$ is the standard Hermitian inner product on $\mathbb{C}\times \mathbb{C}^n$ and $\mbox{grad} f(\gamma(s))$ is the gradient vector of $f$ at $\gamma(s)$, that is,
\begin{equation*}
\mbox{grad} f(\gamma(s)):=\Biggl( \overline{\frac{\partial f}{\partial t}(\gamma(s))}, \overline{\frac{\partial f}{\partial z_1}(\gamma(s))},\ldots, \overline{\frac{\partial f}{\partial z_n}(\gamma(s))} \Biggr),
\end{equation*}
where the bar stands for the complex conjugation.
Let
\begin{equation*}
\mbox{ord}\, \ell(s):=\inf_{0\leq i\leq n}\, \mbox{ord}\, \ell_i(s)
\end{equation*}
where $\mbox{ord}\, \ell_i(s)$ is the order (in $s$) of the i-th component $\ell_i(s)$ of $\ell(s)$.  Clearly,
$\mbox{ord}\, \ell(s)\leq w_{m+1}$.
Note that if $\mbox{ord}\, \ell(s)<w_{m+1}$, then 
\begin{equation}\label{liml1}
\lim_{s\to 0}\frac{\ell(s)}{\vert s\vert^{\mbox{\tiny ord}\, \ell(s)}} = 
(*,\underbrace{*,\ldots,*}_{m \mbox{ \tiny terms}},\underbrace{0,\ldots,0}_{n-m \mbox{ \tiny zeros}}),
\end{equation}
where each term marked with a star ``$*$'' represents a complex number which may be zero or not.
On the other hand, if $\mbox{ord}\, \ell(s)=w_{m+1}$, then
\begin{equation}\label{liml2}
\lim_{s\to 0}\frac{\ell(s)}{\vert s\vert^{\mbox{\tiny ord}\, \ell(s)}} = (*,\underbrace{*,\ldots,*}_{m \mbox{ \tiny terms}},a_{m+1},\ldots,a_{m+m_1},\underbrace{0,\ldots,0}_{n-m -m_1\mbox{ \tiny zeros}}).
\end{equation}
Let $\mathbf{w}:=(w_1,\ldots,w_n)=(0,\ldots,0,w_{m+1},\ldots,w_n)$, and let $l_\mathbf{w}\colon \Gamma\to \mathbb{R}$ be the restriction of the linear map 
\begin{equation*}
(x_1,\ldots, x_n)\in \mathbb{R}^n\mapsto\sum_{1\leq i\leq n} x_i w_i\in \mathbb{R}, 
\end{equation*}
where $\Gamma$ is the non-compact Newton boundary $\Gamma_{nc}(f_t;\mathbf{z})$, which is independent of $t$. Denote by $d_\mathbf{w}$ the minimal value of $l_\mathbf{w}$, and write $\Delta_\mathbf{w}$ for the face of $\Gamma$ defined by the locus where $l_\mathbf{w}$ takes this minimal value. Clearly, $\Delta_\mathbf{w}$ is an essential non-compact face, and $I_{\Delta_\mathbf{w}}=I$.
Finally, let $\mathbf{a}:=(a_1,\ldots,a_n)$. 
As $\Gamma$ does not depend on $t$, 
\begin{align}\label{f1}
\frac{\partial f}{\partial z_i} (\gamma(s)) = \frac{\partial {(f_{\tau})}_{\Delta_{\mathbf{w}}}}{\partial z_i} (\mathbf{a})\, s^{d_\mathbf{w}-w_i}+\cdots 
\end{align}
for any $1\leq i\leq n$, while
\begin{align}\label{ldt0}
\lim_{s\to 0}\biggl(\frac{1}{\vert s\vert^{d_{\mathbf{w}}-1}}\cdot\frac{\partial f}{\partial t} (\gamma(s))\biggr) = 0.
\end{align}
(As usual, ${(f_{\tau})}_{\Delta_{\mathbf{w}}}$ is the face function associated with $f_{\tau}$ and $\Delta_{\mathbf{w}}$.)
Also, note that, by (\ref{assumptionwi}), 
\begin{equation}\label{ineg2}
\begin{aligned}
d_\mathbf{w}-w_{m+1}=\cdots=d_\mathbf{w}-w_{m+m_1}>d_\mathbf{w}-w_{m+m_1+1}=\cdots\\
\cdots=d_\mathbf{w}-w_{m+m_1+m_2}>\cdots\ d_\mathbf{w}-w_n.
\end{aligned}
\end{equation}
Let $o(s):=\mbox{ord}(\overline{\mbox{grad} f(\gamma(s))})$, that is,
\begin{equation*}
o(s):=\inf\biggl\{ \mbox{ord}\biggl(\frac{\partial f}{\partial t} (\gamma(s))\biggr), \inf_{1\leq i\leq n} \mbox{ord} \biggl(\frac{\partial f}{\partial z_i} (\gamma(s))\biggr) \biggr\}.
\end{equation*}
By the uniform local tameness (i.e., the condition $r_{nc}(f_t)\geq \rho$ for all small $t$), if $(\tau,\mathbf{q})$ is close enough to $(0,\mathbf{0})\in\mathbb{C}\times\mathbb{C}^n$, then there exists an integer $i_0\in\{m+1,\ldots,n\}$ such that
\begin{equation}\label{uult}
\frac{\partial (f_\tau)_{\Delta_{\mathbf{w}}}}
{\partial z_{i_0}}(\mathbf{a}) \not= 0.
\end{equation}
Combined with (\ref{f1}), the relation (\ref{uult}) shows that 
\begin{equation*}
o(s)\leq d_{\mathbf{w}}-w_{i_0}\overset{\tiny \mbox{(\ref{ineg2})}}\leq d_{\mathbf{w}}-w_{m+1}\leq d_{\mathbf{w}}-1. 
\end{equation*}
Then, by (\ref{ldt0}),
\begin{align*}
\lim_{s\to 0}\biggl(\frac{1}{\vert s\vert^{o(s)}}\cdot\frac{\partial f}{\partial t} (\gamma(s))\biggr) = 0,
\end{align*}
and since $w_i=0$ for $1\leq i\leq m$,
\begin{equation}\label{mr}
\lim_{s\to 0}\frac{\mbox{grad} f(\gamma(s))}{\vert s\vert^{o(s)}} = 
\left\{
\begin{aligned}
& (0,\underbrace{0,\ldots,0}_{m+m_1 \mbox{ \tiny zeros}},\underbrace{*,\ldots,*}_{\ n-m-m_1 \mbox{ \tiny terms}}) \quad \mbox{if}\quad  o(s)<d_{\mathbf{w}}-w_{m+1},\\
& \biggl(0,\underbrace{0,\ldots,0}_{m\mbox{ \tiny zeros}},\overline{\frac{\partial {(f_\tau)}_{\Delta_{\mathbf{w}}}}{\partial z_{m+1}} (\mathbf{a})},\ldots, \overline{\frac{\partial {(f_\tau)}_{\Delta_{\mathbf{w}}}}{\partial z_{m+m_1}} (\mathbf{a})},\underbrace{*,\ldots,*}_{n-m-m_1 \mbox{ \tiny terms}}\biggr)\\ 
& \mbox{\hskip 4.4cm if}\quad o(s)=d_{\mathbf{w}}-w_{m+1}.
\end{aligned}
\right.
\end{equation}
Since $\Vert \ell(s) \Vert \sim c_1 \vert s \vert^{\mbox{\tiny ord}\, \ell(s)}$ and $\Vert \mbox{grad} f(\gamma(s)) \Vert \sim c_2 \vert s \vert^{o(s)}$ as $s\to 0$ ($c_1$, $c_2$ constants), it follows immediately from (\ref{liml1}), (\ref{liml2}) and (\ref{mr}) that the relation (\ref{lfwbrc}) is satisfied if $o(s)<d_{\mathbf{w}}-w_{m+1}$ or if $o(s)=d_{\mathbf{w}}-w_{m+1}$ and $\mbox{ord}\, \ell(s)<w_{m+1}$. In order to show that (\ref{lfwbrc}) also holds when $o(s)=d_{\mathbf{w}}-w_{m+1}$ and $\mbox{ord}\, \ell(s)=w_{m+1}$, we must prove that
\begin{equation}\label{whp}
\sum_{i=m+1}^{m+m_1} a_i \frac{\partial {(f_\tau)}_{\Delta_{\mathbf{w}}}}{\partial z_{i}} (\mathbf{a})=0.
\end{equation}
To prove (\ref{whp}), we proceed as follows. 
The polynomial $(f_\tau)_{\Delta_{\mathbf{w}}}$ is weighted homogeneous with respect to the weights $\mathbf{w}$ and has weighted degree $d_{\mathbf{w}}$. Then, by the Euler identity, for any $\mathbf{z}=(z_1,\ldots,z_n)$ we have:
\begin{equation}\label{euler}
\sum_{1\leq i\leq n} w_i \, z_i \, \frac{\partial {(f_\tau)}_{\Delta_{\mathbf{w}}}}{\partial z_i}(\mathbf{z})=d_{\mathbf{w}} \cdot {(f_\tau)}_{\Delta_{\mathbf{w}}}(\mathbf{z}).
\end{equation}
As $f(\gamma(s))=0$ for any $s$, we have $({f_\tau})_{\Delta_{\mathbf{w}}}(\mathbf{a})=0$. Therefore, by (\ref{euler}), 
\begin{equation}\label{adlt2}
\sum_{1\leq i\leq n} w_i \, a_i \frac{\partial {(f_\tau)}_{\Delta_{\mathbf{w}}}}{\partial z_i}(\mathbf{a})=0.
\end{equation}
Combined with (\ref{f1}) and (\ref{ineg2}), the equality $o(s)=d_{\mathbf{w}}-w_{m+1}$ implies that for any $i>m+m_1$,
\begin{equation}\label{adlt}
\frac{\partial {(f_\tau)}_{\Delta_{\mathbf{w}}}}{\partial z_i} (\mathbf{a})=0.
\end{equation}
Indeed, if there were $i_1> m+m_1$ such that (\ref{adlt}) does not hold, then, by (\ref{f1}),
\begin{equation*}
o(s)\leq d_{\mathbf{w}}-w_{i_1}\overset{\tiny \mbox{(\ref{ineg2})}}<d_{\mathbf{w}}-w_{m+1},
\end{equation*}
which is a contradiction.
As $w_i=0$ for $1\leq i\leq m$, it follows from (\ref{adlt2}) and (\ref{adlt}) that
\begin{equation*}
\sum_{i=m+1}^{m+m_1} w_i \, a_i \frac{\partial {(f_\tau)}_{\Delta_{\mathbf{w}}}}{\partial z_i}(\mathbf{a}) = w_{m+1} \sum_{i=m+1}^{m+m_1} a_i \frac{\partial {(f_\tau)}_{\Delta_{\mathbf{w}}}}{\partial z_i}(\mathbf{a})=0.
\end{equation*}
As $w_{m+1}>0$, the equality (\ref{whp}) follows. This completes the proof.

\begin{remark}\label{rk-Ivide}
In the above proof, we have assumed $I\not=\emptyset$ (i.e., $m\geq 1$). However, a straightforward modification shows that the argument still works when $I=\emptyset$. We just observe that, in this case, the face $\Delta_{\mathbf{w}}$ is compact, and hence, in (\ref{uult}), instead of the uniform local tameness, it suffices to invoke the non-degeneracy condition. 
\end{remark}

\begin{remark}\label{rk-rbt}
Note that in the special case where the functions $f_t$ have an \emph{isolated} singularity at the origin, we recover the Brian\c con assertion mentioned in the introduction (cf.~Assertion \ref{thm-briancon}). Indeed, in this case, as observed by Brian\c con in \cite{B}, by adding monomials of the form $z_i^{N_i}$ for large values of $N_i$, we may assume that $f_t$ is \emph{convenient}, that is, the intersection of the Newton boundary $\Gamma(f_t;\mathbf{z})$ with each coordinates subspace is non-empty. (In other words, the family of hypersurfaces $\{V(\tilde f_t)\}$ defined by the polynomial functions $\tilde f_t(\mathbf{z}):=f_t(\mathbf{z})+c_1z_1^{N_1}+\cdots+c_nz_n^{N_n}$ is Whitney equisingular if and only if the original family $\{V(f_t)\}$ is.) But for convenient polynomials, the only vanishing coordinates subspace is $\mathbb{C}^\emptyset=\{\mathbf{0}\}$, and by Remark \ref{rk-Ivide}, the partition 
\begin{equation*}
\{A_I\, ;\, I\in\mathscr{I}_{nv}(f_0)\}\cup \{C_\emptyset\}
\end{equation*}
of $\mathscr{U}\cap V(f)$ is a Whitney stratification with the $t$-axis as a stratum if $\Gamma(f_t;\mathbf{z})$ is independent of $t$ and $f_t$ is non-degenerate.
\end{remark}

\section{Examples of admissible families}

In this section, we give some examples of admissible (and therefore Whitney equisingular) families with non-isolated singularities.

\subsection{A family of curves with non-isolated singularities} Consider the family given by the polynomial function 
\begin{equation*}
f(t,z_1,z_2):=z_1^2z_2^3+z_1^3z_2^2+tz_1^2z_2^4.
\end{equation*}
Let $A=(2,3)$ and $B=(3,2)$. For $t$ small enough, the singular locus of $f_t$ in a~suf\-ficiently small open neighbourhood of the origin in $\mathbb{C}^2$ consists of the coordinates axes.
The non-compact Newton boundary $\Gamma_{nc}(f_t;\mathbf{z})$, which is clearly independent of $t$, has one compact $1$-dimensional face (namely, the face $\overline{AB}$), two $0$-dimensional faces ($A$ and $B$), and two essential non-compact faces: $\Xi_1:=B+\mathbb{R}_+\mathbf{e}_1$ and $\Xi_2:=A+\mathbb{R}_+\mathbf{e}_2$.  We easily check that for each compact face $\Delta$ and each $t$, the face function $(f_t)_{\Delta}$ has no critical point on $(\mathbb{C}^*)^2$ (i.e., $f_t$ is non-degenerate).
We claim that for any $I\in\mathscr{I}_v(f_0)=\mathscr{I}_v(f_t)$, the family $\{f_t\}$ is uniformly locally tame along $\mathbb{C}^I$. Indeed, a trivial calculation shows that for any fixed $u_1\in\mathbb{C}^*$, the function 
\begin{equation*}
z_2\mapsto(f_t)_{\Xi_1}(u_1,z_2):=u_1^3z_2^2
\end{equation*}
of the variable $z_2$ has no critical point on $\mathbb{C}^{*\{1,2\}}_{u_1}$.
Similarly, for any fixed $u_2\in\mathbb{C}^*$ with $\vert u_2\vert<1/\vert t\vert$ (if $t\not=0$), the function 
\begin{equation*}
z_1\mapsto(f_t)_{\Xi_2}(z_1,u_2):=z_1^2u_2^3+tz_1^2u_2^4
\end{equation*}
of the variable $z_1$ has no critical point on $\mathbb{C}^{*\{1,2\}}_{u_2}$. So we can take
\begin{equation*}
r_{nc}(f_t)=\left\{
\begin{aligned}
1/\vert t\vert\mbox{ for } t\not=0,\\
\infty \mbox{ for } t=0,
\end{aligned}
\right.
\end{equation*}
and we have $r_{nc}(f_t)>\rho:=1$ for all $t$ with $\vert t\vert<1$. It follows that
the family $\{f_t\}$ is admissible. 

\subsection{Families with ``big'' exponents for $t$-dependent monomials}
Let $h(\mathbf{z})$ be a polynomial function on $\mathbb{C}^n$ and let $g(t,\mathbf{z})$ be a polynomial function on $\mathbb{C}\times \mathbb{C}^n$. As usual, we write $g_t(\mathbf{z}):=g(t,\mathbf{z})$. Suppose that for all small $t$, $\Gamma_+(g_t;\mathbf{z})\subseteq\Gamma_+(h;\mathbf{z})$ and $\Gamma_{nc}(g_t;\mathbf{z})\cap\Gamma_{nc}(h;\mathbf{z})=\emptyset$. Under this assumption, if $h$ is non-degenerate and locally tame along the vanishing coordinates subspaces, then the family $\{f_t\}$ defined by $f_t(\mathbf{z}):=h(\mathbf{z})+g_t(\mathbf{z})$ is admissible. For example, the family given by $f(t,z_1,z_2):=z_1^2z_2^3+z_1^3z_2^2+tz_1^3z_2^3$ is admissible.

\subsection{Admissible families and branched coverings} Take a positive integer $p\in \mathbb{N}^*$, and consider the branched covering 
\begin{equation*}
\varphi^{p}\colon \mathbb{C}^n\to\mathbb{C}^n,\,
(z_1,\ldots, z_n)\mapsto (z_1^{p},\ldots, z_n^{p}),
\end{equation*}
whose ramification locus is given by the coordinates hyperplanes $z_i=0$ ($1\leq i\leq n$). In \cite[Proposition 22]{O1}, the second author showed that if $h(\mathbf{z})$ is a non-degenerate polynomial function which is locally tame along $\mathbb{C}^I$ for any $I\in\mathscr{I}_v(h)$, then $h^{p}(\mathbf{z}):=h\circ \varphi^{p}(\mathbf{z})=h(z_1^{p},\ldots, z_n^{p})$ is non-degenerate, $\mathscr{I}_v(h^{p})=\mathscr{I}_v(h)$, and $h^{p}$ is locally tame along $\mathbb{C}^I$ for any $I\in\mathscr{I}_v(h^{p})$. Actually, the proof shows that if $\{f_t\}$ is an admissible family of polynomial functions, then so is the family $\{f_t^{p}\}$, where $f_t^{p}(\mathbf{z}):=f_t\circ \varphi^{p}(\mathbf{z})$. Indeed, it is not difficult to see that the independence of $\Gamma_{nc}(f_t;\mathbf{z})$ with respect to $t$ implies that of $\Gamma_{nc}(f_t^{p};\mathbf{z})$. Also, it is easy to check that $\mathscr{I}_v(f_t^{p})=\mathscr{I}_v(f_t)$ and $\mathscr{I}_{nv}(f_t^{p})=\mathscr{I}_{nv}(f_t)$. To see that $f_t^{p}$ is non-degenerate, we argue by contradiction. Take any compact face $\Xi\subseteq \Gamma(f_t^{p};\mathbf{z})$, and suppose there exists $\mathbf{a}:=(a_1,\ldots,a_n)\in(\mathbb{C}^*)^n$ such that  
\begin{equation}\label{degface}
\frac{\partial (f_t^{p})_\Xi}{\partial z_i}(\mathbf{a})=0
\end{equation}
for all $1\leq i\leq n$. Clearly, $(f_t^{p})_\Xi(\mathbf{z})=(f_t)_\Delta\circ\varphi^{p}(\mathbf{z})=(f_t)_\Delta(z_1^{p},\ldots,z_n^{p})$, where $\Delta$ is the compact face of $\Gamma(f_t;\mathbf{z})$ corresponding to $\Xi$---that is, if $\Delta\cap\mathbb{N}^2=\{(\alpha_1,\alpha_2),(\beta_1,\beta_2),\ldots\}$, then $\Xi\cap\mathbb{N}^2=\{(\alpha_1^p,\alpha_2^p),(\beta_1^p,\beta_2^p),\ldots\}$ (cf.~Figure \ref{figure2}). Therefore, by (\ref{degface}),
\begin{equation*}
\frac{\partial (f_t)_\Delta}{\partial z_i}(a_1^{p},\ldots,a_n^{p})\cdot p a_i^{p-1}=0,
\end{equation*}
and hence,
\begin{equation*}
\frac{\partial (f_t)_\Delta}{\partial z_i}(a_1^{p},\ldots,a_n^{p})=0
\end{equation*}
for all $1\leq i\leq n$. This contradicts the non-degeneracy of $f_t$. 

\begin{figure}[t]
\includegraphics[width=16cm,height=5cm]{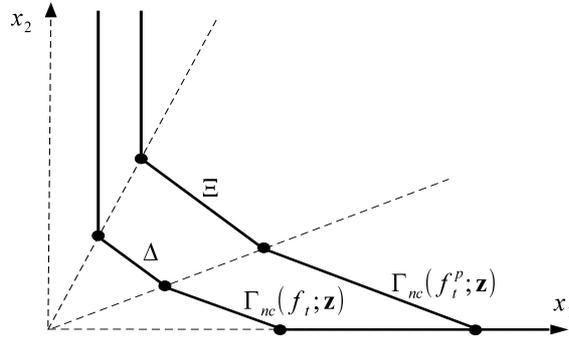}
\caption{A face $\Delta\subseteq\Gamma_{nc}(f_t;\mathbf{z})$ and its corresponding face $\Xi\subseteq\Gamma_{nc}(f_t^p;\mathbf{z})$}
\label{figure2}
\end{figure}

\begin{claim} 
The family $\{f_t^{p}\}$ is uniformly locally tame along the vanishing coordinates subspaces.
\end{claim}

\begin{proof}
By hypothesis, we know that the family $\{f_t\}$ is uniformly locally tame along the vanishing coordinates subspaces---that is, $r_{nc}(f_t)\geq \rho$ for all small $t$. Without loss of generality, we may assume that $\rho<1$. Consider a subset $I\in\mathscr{I}_v(f_t^{p})$. For simplicity, let us assume that $I=\{1,\ldots,m\}$. Let $u_1,\ldots,u_m$ be non-zero complex numbers such that 
\begin{equation*}
\vert u_1 \vert^2 + \cdots + \vert u_m \vert^2\leq \rho^2.
\end{equation*}
Take any essential non-compact face $\Xi\in\Gamma_{nc}(f_t^{p};\mathbf{z})$ such that $I_\Xi=I$, and consider the corresponding face $\Delta\in\Gamma_{nc}(f_t;\mathbf{z})$. (Note that $I_\Delta=I$ too.) We want to show that the face function $(f_t^{p})_\Xi$ has no critical point on $\mathbb{C}^{*\{1,\ldots,n\}}_{u_1,\ldots,u_m}$ as a function of the variables $z_{m+1},\ldots, z_n$. Again we argue by contradiction. Suppose $(u_1,\ldots,u_m,a_{m+1},\ldots,a_n)$ is a critical point.
Then, for $m+1\leq i\leq n$,
\begin{align*}
0 & = \frac{\partial (f_t^{p})_\Xi}{\partial z_i}(u_1,\ldots,u_m,a_{m+1},\ldots,a_n) \\
& = \frac{\partial (f_t)_\Delta}{\partial z_i}(u_1^{p},\ldots,u_m^{p},a_{m+1}^{p},\ldots,a_n^{p})\cdot p a_i^{p-1},
\end{align*}
and therefore,
\begin{align}\label{getcontradiction}
\frac{\partial (f_t)_\Delta}{\partial z_i}(u_1^{p},\ldots,u_m^{p},a_{m+1}^{p},\ldots,a_n^{p})=0.
\end{align}
As $\rho<1$,
\begin{equation*}
\vert u_1^p \vert^2 + \cdots + \vert u_m^p \vert^2\leq \rho^2,
\end{equation*}
and hence (\ref{getcontradiction}) contradicts the uniform local tameness of the family $\{f_t\}$. 
\end{proof}

Altogether, the family $\{f_t^{p}\}$ is admissible.

\bibliographystyle{amsplain}

\end{document}